\documentclass[10pt]{amsart}
\theoremstyle{plain}

\newtheorem{theorem}{Theorem}
\newtheorem{lemma}[theorem]{Lemma}
\newtheorem{proposition}[theorem]{Proposition}
\newtheorem{corollary}[theorem]{Corollary}

\theoremstyle{definition}

\newtheorem{remark}{Remark}

\newcommand{\sgn}{\operatorname{sgn}}
\newcommand{\spt}{\operatorname{spt}}

\newcommand{\cov}{\operatorname{cov}}

\newcommand{\nor}{\operatorname{nor}}
\newcommand{\Tan}{\operatorname{Tan}}
\newcommand{\Nor}{\operatorname{Nor}}
\newcommand{\Lin}{\operatorname{Lin}}
\newcommand{\const}{\operatorname{const}}

\newcommand{\R}{\mathbb{R}}

\newcommand{\Z}{\mathbb{Z}}

\newcommand{\Ha}{\mathcal{H}}
\newcommand{\cL}{\mathcal{L}}
\newcommand{\cE}{\mathcal{E}}
\newcommand{\cF}{\mathcal{F}}

\newcommand{\1}{{\bf 1}}

\newcommand{\reach}{{\rm reach}\,}

\newcommand{\E}{{\mathbb E}}

\newcommand{\relint}{\operatorname{rel\,int}}
\usepackage{amssymb}
\usepackage{amsmath}
\usepackage{amsthm}

\date{\today}

\begin{document}

\title{Mean Euler characteristic of stationary random closed sets}

\author{Jan Rataj}

\address{Charles University\\
Faculty of Mathematics and Physics\\
Sokolovsk\'a 83\\
186 75 Praha 8\\
Czech Republic}

\email{rataj@karlin.mff.cuni.cz}

\thanks{The research was supported by the Czech Science Foundation, Project No.~18-11058S, and by Charles University, Project PRIMUS/17/SCI/3.}

\begin{abstract}
The translative intersection formula of integral geometry yields an expression for the mean Euler characteristic of a stationary random closed set intersected with a fixed observation window. We formulate this result in the setting of sets with positive reach and using flag measures which yield curvature measures as marginals. As an application, we consider excursion sets of stationary random fields with $C^{1,1}$ realizations, in particular, stationary Gaussian fields, and obtain results which extend those known from the literature.
\end{abstract}

\maketitle

\section{Introduction}
Translative and kinematic intersection formulas for random closed sets play an important role in stochastic geometry. We refer to \cite{SW08}, namely Parts II, III, as a basic reference source. Since we are interested in (total) curvature measures, in particular, Euler characteristic, of the intersection of a stationary random closed set in $\R^d$ with a fixed observation window $W$, some regularity assumptions are needed. While the basic model considered in \cite{SW08} is that of the extended convex ring, we will work with sets of locally positive reach (i.e., closed subsets of $\R^d$ with positive reach at any its point) instead. One basic advantage is that sets with $C^2$ (or even $C^{1,1}$) boundaries are included. The basic reference for random closed sets with positive reach is \cite{Z84}.

Let $Z\subset\R^d$ be a stationary random closed set with positive reach. Then, given $0\leq k\leq d-1$, the $k$th curvature measure, $C_k(Z,\cdot)$, is a random signed measure and its mean is translation invariant, hence, provided that
\begin{equation}  \label{fin_dens}
\E |C_k|(Z,\cdot)\text{ is locally finite},\quad k=0,\dots, d
\end{equation}
($|C_k|(Z,\cdot)$ denotes the total variation of $C_k(Z,\cdot)$), it must be a multiple of the Lebesgue measure:
$$\E C_k(Z,\cdot)=\overline{V}_k(Z)\, \cL^d(\cdot),$$
and $\overline{V}_k(Z)$ is called the $k$th \emph{curvature density} of $Z$. We define also $\overline{V}_d(Z)$ by $\E\cL^d(Z\cap\cdot)=\overline{V}_d(Z)\cL^d(\cdot)$ ($\cL^d$ stands for the Lebesgue measure in $\R^d$).

If $Z$ were moreover isotropic, one could apply the principal kinematic formula (\cite[Theorem~6.1]{RZ19}) and obtain for any $W\subset\R^d$ with positive reach 
\begin{equation} \label{PKF}
\E \chi(Z\cap W)=\sum_{k=0}^d \beta_{d,k}\overline{V}_k(Z)C_{d-k}(W)
\end{equation}
with constants 
$$\beta_{d,k}=\frac{\Gamma\left(\frac{k+1}2\right) \Gamma\left(\frac{d-k+1}2\right)}{\Gamma\left(\frac{d+1}2\right)\Gamma\left(\frac{1}2\right)}$$ 
($\chi$ denotes the Euler-Poincar\'e characteristic and $C_k(W):=C_k(W,\R^d)$ is the total $k$th curvature measure of $W$). If, however, only stationarity is assumed, we have to use the translative intersection formula instead (see \cite{RZ95} or \cite[Ch.~6]{RZ19} for the case of sets with positive reach). But then, the right hand side does not factorize into products of curvatures of the two separate sets, but includes certain mixed functionals of both sets. These mixed functionals can be, however, under certain integrability assumption, represented as product integrals of a given function with respect to so called \emph{flag measures} of the two sets, see \cite{HRW13} for the case of convex bodies (the case of sets with positive reach follows directly, as shown in Section~2). For $k-0,\dots,d-1$, the $k$th flag measure of $Z$, $\Gamma_k(Z,\cdot)$, is a signed random measure on the product $\R^d\times F^\perp(d,k^*)$, where $k^*:=d-1-k$ and
$$F^\perp(d,j):=\{(u,U):\, u\in S^{d-1},\, U\in G(d,j),\, u\perp U\}$$
is the $j$th \emph{flag space} (here $S^{d-1}$ denotes the unit sphere and $G(d,k)$ the Grassmannian of $j$-dimensional linear subspaces of $\R^d$). The marginal measure $\Gamma_k(Z,(\cdot)\times F^\perp(d,k^*))$ agrees with $C_k(Z,\cdot)$ and again, if $Z$ is stationary, the mean value $\E\Gamma_k(Z,\cdot)$ disintegrates, provided it is locally finite, as follows:
\begin{equation} \label{E_disint2}
\E\int g(x,u,U)\,\Gamma_k(Z,d(x,u,U))=\int_{\R^d}\int_{F^\perp(d,k^*)}g(x,u,U)\, \overline{\Omega}_k(Z,d(u,U))\, dx
\end{equation}
for any nonnegative measurable function $g$ on $\R^d\times F^\perp(d,k^*)$, with a finite signed measure $\overline{\Omega}_k(Z,\cdot)$ called $k$th \emph{specific flag measure} of $Z$. For a \emph{compact} set $W\subset\R^d$ with positive reach, $\Omega_k(W,\cdot)$ is simply the marginal of $\Gamma_k(W,\cdot)$ on $F^\perp(d,k)$. Further, the marginal measures of $\overline{\Omega}_k(Z,\cdot)$, $\Omega_k(W,\cdot)$ on the unit sphere $S^{d-1}$ are denoted by $\overline{S}_k(Z,\cdot)$, $S_k(W,\cdot)$, and they are called $k$th \emph{specific area measure} of $Z$, \emph{area measure} of $W$, respectively.

We say that two sets $X,Y\subset\R^d$ with locally positive reach \emph{touch} if there exists pairs in the unit normal bundles $(x,u)\in\nor X$, $(x,-u)\in\nor Y$, see Section~2 for details. Note that the intersection $X\cap Y$ has locally positive reach provided that $X,Y$ do not touch; otherwise, the positive reach of the intersection is not guaranteed.

Our first main result follows. The function $F_k(\theta)$ of an angle $\theta\in[0,\pi]$, as well as the continuous function $\varphi_k$ on $F^\perp(d,k^*)\times F^\perp(d,(d-k)^*)$, are introduced in Section~2. By $|\mu|$ we denote the total variation of a signed measure $\mu$.

\begin{theorem}  \label{T1}
Let $Z\subset\R^d$ be a stationary random closed set with realizations of locally positive reach and satisfying \eqref{fin_dens}. Let $W\subset\R^d$ be a compact set with positive reach and assume that 
\begin{equation} \label{touch2}
\Pr[Z\text{ and }W\text{ touch}]=0,
\end{equation}
and  
\begin{equation}  \label{integr_E}
F_k(\angle(u,v))|\varphi_k(u,U,v,V)|\text{ is } \E|\Omega_k|(Z,\cdot)\otimes|\Omega_{d-k}|(W,\cdot)\text{-integrable}.
\end{equation}
Then
\begin{align*}
\E \chi(Z\cap W)=&\overline{V}_0(Z)\cL^d(W)+\overline{V}_d(Z)\chi(W)\\
+&\sum_{k=1}^{d-1}\int\int F_k(\angle(u,v))\, \varphi_k(u,U,v,V)\, \Omega_{d-k}(W,d(v,V))\, \overline{\Omega}_k(Z,d(u,U)).
\end{align*}
A sufficient condition for \eqref{integr_E} is that
\begin{equation}   \label{suffic}
\E \int\int \sin^{3-d}\angle(u,v)\, |S_k|(Z,du)\, |S_{d-k}|(W,dv)<\infty.
\end{equation}
Further, \eqref{suffic} holds whenever $\E|S_k|(Z,\cdot)$ has bounded density w.r.t. uniform distribution on $S^{d-1}$.
\end{theorem}

\begin{remark}
A related formula for stationary random sets from the extended convex ring can be found in \cite[Theorem~9.4.1]{SW08}, where an integrability condition concerning the local number of its convex components is assumed, but the right-hand side is expressed as a sum of mixed functionals of $Z$ and $W$. Similar formulas can be derived in the setting of sets with positive reach and under integrability condition \eqref{fin_dens}. The main point of Theorem~\ref{T1} is that the mixed functionals are desintegrated using the flag measures, as shown in \cite{HRW13}. Such a disintegration is possible only under assumption \eqref{integr_E} which can be considered as a kind of general mutual position of $Z$ and $W$; if both $Z$ and $W$ contain parallel faces of sufficiently high dimension then this assumption can be violated, as \cite[Section~8]{HRW13} shows. If $d\leq 3$ then \eqref{suffic} (and, hence, also \eqref{integr_E}) is always satisfied.
\end{remark}

We will apply Theorem~\ref{T1} to excursion sets of stationary random fields with smooth realizations. If $\xi$ is a stationary random field in $\R^d$ which is $C^{1,1}$ almost surely and if $\alpha\in\R$ is its \emph{regular} value (i.e., $\nabla\xi(x)\neq 0$ whenever $\xi(x)=\alpha$) almost surely, then
$$Z_\alpha:=\{x\in\R^d:\, \xi(x)\geq\alpha\}$$
is a stationary random closed set.

We are mainly (but not only) interested in the case when $\xi$ is \emph{Gaussian} with zero mean; then the distribution of $\xi$ is characterized by its covariance function $C(x)=\E\xi(0)\xi(x)$, $x\in\R^d$. In order that $\xi$ be $C^{1,1}$ almost surely, it is necessary that the second order partial derivatives of $\xi$ exist in the $L^2$ sense (otherwise, the first order derivatives would be nondifferentiable almost everywhere and almost surely). This means that the covariance function $C$ has to be of class $C^4$. We refer to \cite[Ch.~2]{A81} and \cite[Ch.~1]{AT07} for smoothness of Gaussian random fields.

Our second main result is the form of the specific flag measure of the excursion sets of stationary centred Gaussian random fields. We need the following notation. If $\Lambda$ is a symmetric $d\times d$ matrix and $U\in G(d,k)$ ($1\leq k\leq d$), we denote
$$\Lambda[U]:=\det\left( \langle\Lambda u_i,u_j\rangle\right)_{i,j=1}^k,$$
where $\{u_1,\dots,u_k\}$ is an orthonormal basis of $U$ and $\langle\cdot,\cdot\rangle$ is the standard scalar product in $\R^d$. (It is easy to see that the determinant does not depend on the choice of the orthonormal basis; see Section~2 for details.)

We will denote by $\mu_k$ the rotation invariant probability measure on $F^\perp(d,k^*)$ which is given by
$$\int_{F^\perp(d,k^*)} g(u,U)\,\mu(d(u,U))=O_{d-1}^{-1}\int_{S^{d-1}}\int_{G^{u^\perp}(d-1,k^*)} g(u,U)\, dU\, \Ha^{d-1}(du),$$
where $O_{d-1}=\Ha^{d-1}(S^{d-1})$ denotes the $(d-1)$-dimensional Hausdorff measure of the unit sphere in $\R^d$.

\begin{theorem}  \label{T2}
Let $\xi$ be a stationary Gaussian random field with zero mean and such that $\xi$ is $C^{1,1}$ almost surely. Assume further that $\nabla\xi(0)$ has nondegenerate distribution. Then any $\alpha\in\R$ is a regular value of $\xi$ almost surely, and for any $0\leq k\leq d-1$, the specific flag measure $\overline{\Omega}_k(Z_\alpha,\cdot)$ of the excursion set $\Z_\alpha=\{\xi\geq\alpha\}$ is absolutely continuous w.r.t. $\mu_k$, with density
$$
q_k(u,U)=\frac{\beta_{d,k}^{-1}\;(2\pi)^{\frac{k-d-1}2}}{O_{d-1}\sigma^{d-k}\sqrt{\det\Lambda}}e^{-\frac{\alpha^2}{2\sigma^2}}
H_{k^*}\left(\frac\alpha\sigma\right)\left(u^T\Lambda^{-1}u\right)^{-\frac k2-1}\, 
\Lambda[U] ,\quad (u,U)\in F^\perp(d,k^*),
$$
where $\sigma^2:=C(0)$, $\Lambda=\left(\E\frac{\partial\xi}{\partial x_i}(0) \frac{\partial\xi}{\partial x_j}(0)\right)_{i,j=1}^d$ and $H_{k^*}$ is the Hermite polynomial of order $k^*$.
\end{theorem}

\begin{remark} As explained at the beginning of Section~\ref{GRF}, the assumption that the paths are $C^{1,1}$ a.s. implies already $C^2$-smoothness a.s. in the case of stationary Gaussian random fields. When relaxing either stationarity, or Gaussianity, examples of random fields which are $C^{1,1}$ and not $C^2$ can be obtained (see, e.g., \cite[Example~12]{LR19}). For excursion sets of such random fields, Theorem~1 can also be applied, but the formula would be less tractable in the general case.
\end{remark}

Inserting the density from Theorem~\ref{T2} into the formula of Theorem~\ref{T1}, we obtain the following

\begin{corollary}\label{C2}
Under the assumptions of Theorem~\ref{T2}, the excursion set $Z_\alpha=\{\xi\geq\alpha\}$ fulfills for any $0\leq k\leq d$ and any compact set $W\subset\R^d$ with positive reach
\begin{align*}
\E \chi(Z_\alpha\cap W)&=\overline{V}_0(Z_\alpha)\cL^d(W)+\overline{V}_d(Z_\alpha)\chi(W)\\
+&\sum_{k=1}^{d-1}\int\int F_k(\angle(u,v))\, \varphi_k(u,U,v,V)\, \Omega_{d-k}(W,d(v,V))\, q_k(u,U)\,\mu_k(d(u,U)).
\end{align*}
\end{corollary}

We are able to simplify the formula from Corollary~\ref{C2} significantly only if $W$ is a \emph{zonotope} (i.e., finite Minkowski sum of segments). We can assume without loss of generality that the segments generate at the origin, i.e., that $0$ is a vertex of $W$. Let $\cF_j^0(W)$ denote the set of all $j$-faces of $W$ containing the origin, $j=0,\dots,d$, and given $F\in\cF_j^0(W)$, set $|F|:=\cL^j(F)$ and let $F_0\in G(d,j)$ be the linear hull of $F$. The following result extends \cite[Theorem~11.7.2]{AT07} concerned with a cuboid $W$ and proved under slightly more restrictive assumptions on $\xi$.

\begin{corollary}  \label{C4}
Let $\xi$ be as in Theorem~2 and $W\subset\R^d$ a zonotope with a vertex at the origin. Then
$$\E\chi(Z_\alpha\cap W)=\sum_{k=0}^{d-1} \frac{1}{(2\pi)^{\frac{d-k+1}2}\sigma^{d-k}}e^{-\frac{\alpha^2}{2\sigma^2}}H_{k^*}\left(\frac\alpha\sigma\right)\sum_{F\in\cF_{d-k}^0(W)} |F| \sqrt{\Lambda[F_0]}+\psi\left(\frac{\alpha}{\sigma}\right),$$
with $\psi(t)=(2\pi)^{-1/2}\int_t^\infty e^{-s^2/2}\, ds$.
\end{corollary}

The $k$th curvature density can be obtained as the total $k$th flag measure:
$$\overline{V}_k(Z_\alpha)=\overline{\Omega}_k(Z_\alpha,F^\perp(d,k^*))=\int_{F^\perp(d,k^*)}q_k(u,U)\,\mu_k(d(u,U)).$$
This quantity is particularly useful for isotropic random fields when it can be inserted in the principal kinematic formula \eqref{PKF}. Note that (2) together with the following Corollary~\ref{C3} yields a formula which agrees with \cite[Corollary~11.7.3]{AT07} in the case when $W$ is a cube.

\begin{corollary} \label{C3}
Let $\xi$ be as in Theorem~\ref{T2} and assume moreover that $\xi$ is isotropic. Then, for any $\alpha\in\R$ and $0\leq k\leq d-1$, the $k$th curvature density of the excursion set $Z_\alpha$ equals
$$\overline{V}_k(Z_\alpha)=\frac{\beta_{d,k}^{-1}\;(2\pi)^{\frac{k-d-1}2}}{\sigma^{d-k}}\lambda^{\frac{d-k}2}e^{-\frac{\alpha^2}{2\sigma^2}}
H_{k^*}\left(\frac\alpha\sigma\right),$$
where $\lambda:=\E\left(\frac{\partial\xi}{\partial x_1}(0)\right)^2$.
\end{corollary}

We would like to stress that our method of proof is completely different from that of Adler and Taylor \cite{AT07}, where the sectioning body (cuboid) $W$ is considered as a stratified manifold and the Euler characteristic of $Z_\alpha$ within $W$ is computed using the Morse formula of differential geometry. In contrary, we are applying purely integral-geometric methods and work with flag measures which extend the curvature measures. Both methods can be applied in a more general framework. We believe that our method base on flag measures has a good potential of further progress in different settings.

Formulas for random fields with $C^{1,1}$ realizations in $\R^2$ have been obtained recently in \cite{LR19}.

\subsection*{Acknowledgement} It is a pleasure to thank D. Pokorn\'y for helpful conversations.

\section{Preliminaries}
Our basic setting is the Euclidean space $\R^d$ with scalar product $\langle\cdot,\cdot\rangle$ and norm $|\cdot|$.
We will also use the norm of simple multivectors: 
$$\|v_1\wedge\dots\wedge v_k\|^2=\left(\det\left(\langle v_i,v_j\rangle\right)_{i,j=1}^k\right)^2,\quad v_1,\dots,v_k\in\R^d.$$
The symbol $G(d,k)$ denotes the Grassmannian of $k$-dimensional linear subspaces of $\R^d$. Elements of $G(d,k)$ can be identified with unit simple multivectors, modulo the sign. Using this convention, we write e.g.\ for $U\in G(d,k)$ and $V\in G(d,l)$ 
$$\|U\wedge V\|^2=\|u_1\wedge\dots\wedge u_k\wedge v_1\wedge\dots \wedge v_l\|^2,$$
where $\{u_1,\dots,u_k\}$, $\{v_1,\dots,v_l\}$ are any orthonormal bases of $U$, $V$, respectively (the squared norm is independent of the choice of the orthonormal bases). If $U,V\in G(d,k)$ are linear subspaces of the same dimension we write
$$\langle U,V\rangle^2:=\|U\wedge V^\perp\|^2=\left(\det\left(\langle u_i,v_j\rangle\right)_{i,j=1}^k\right)^2.$$
(In fact, $\langle U,V\rangle$ is a scalar product in the space of $k$-vectors.)

If $L:\R^d\to\R^d$ is a linear mapping and $U\in G(d,k)$, we denote by $L_U$ the restriction of $p_U\circ L$ to $U$ ($p_U$ is the orthogonal projection to $U$); hence, $L_U$ is a linear mapping from $U$ to $U$. Note that $L_U$ is selfadjoint whenever $L$ is. We set
$$L[U]:=\det L_U = \det\left(\langle Lu_i,u_j\rangle\right)_{i,j=1}^k,$$
where $\{u_1,\dots,u_k\}$ is an orthonormal basis of $U$ (of course, the determinant is independent of the choice of the orthonormal basis).

\begin{lemma} \label{L_bilin}
Let $L:\R^d\to\R^d$ be linear selfadjoint with eigenvalues $\lambda_i$ and corresponding eigenvectors $b_i$, $i=1,\dots,d$. Then for any $U\in G(d,k)$,
$$L[U]=\sum_{|I|=k}\lambda_I\langle U,B_I\rangle^2,$$
where $\Lambda_I:=\prod_{i\in I}\lambda_i$, $B_I:=\Lin\{b_i:\, i\in I\}$ and the summation is taken over all index sets $I\subset\{1,\dots,d\}$ of cardinality $k$.

In particular, if $L$ is regular and $U=u^\perp$ is the $(d-1)$-subspace perpendicular to a unit vector $u$, we obtain 
\begin{equation}  \label{E_u^perp}
L[u^\perp]=(\det L) \, u^T\Lambda^{-1}u.
\end{equation}
\end{lemma}

\begin{proof}
Let $(u_1,\dots,u_k)$ be an orthonormal basis of $U$. We have
\begin{align*}
L[U]&=
\det \left(\langle Lu_i,u_j\rangle\right)_{i,j=1}^k
=\det\left( \sum_{l=1}^d\lambda_l\langle u_i,b_l\rangle\langle u_j,b_l\rangle\right)_{i,j=1}^k\\
&=\sum_{\sigma\in\Sigma(k)}(\sgn\sigma)\prod_{i=1}^k
\sum_{l=1}^d\lambda_l\langle u_i,b_l\rangle\langle u_{\sigma(i)},b_l\rangle\\
&=\sum_{l_1=1}^d\dots\sum_{l_k=1}^d\lambda_{l_1}\dots\lambda_{l_k}\sum_{\sigma\in\Sigma(k)}(\sgn\sigma)\prod_{i=1}^k\langle u_i,b_{l_i}\rangle\langle u_{\sigma(i)},b_{l_i}\rangle.\\
\end{align*}
If in the multiple sum $l_i=l_j$ for some $1\leq i\neq j\leq k$, the corresponding terms vanishes (since the corresponding terms for different permutations interchanging $i$ and $j$ cancel out). Thus the sum reduces to $k$-tuples of different $l_i$'s, and we can express it as a sum over subsets of $k$ elements and over permutations of these subsets:
$$
L[U]=\sum_{|I|=k}\lambda_I \sum_{\sigma\in\Sigma(k)} \sum_{\pi\in\Sigma(k)}(\sgn\sigma) \prod_{i=1}^k\langle u_i,b_{l_{\pi(i)}}\rangle\langle u_{\sigma(i)},b_{l_{\pi(i)}}\rangle.
$$
Applying the substitutions $j=\pi(i)$, $\sigma_1=\pi^{-1}$ and $\sigma_2=\sigma\circ\pi^{-1}$ (with $\sgn\sigma_2=\sgn\sigma\sgn\pi$), we get
\begin{align*}
L[U]&=\sum_{|I|=k}\lambda_I \sum_{\sigma_1\in\Sigma(k)} \sum_{\sigma_2\in\Sigma(k)}(\sgn\sigma_1)(\sgn\sigma_2) 
\prod_{j=1}^k\langle u_{\sigma_1(j)},b_{l_j}\rangle\langle u_{\sigma_2(j)},b_{l_j}\rangle\\
&=\sum_{|I|=k}\lambda_I\left(\det \left( \langle b_i,u_j\rangle\right)_{i\in I,j\leq k}\right)^2
=\sum_{|I|=k}\lambda_I \left\langle U, B_I\right\rangle^2.
\end{align*}
If $L$ is regular, $U=u^\perp$ and $I=\{1,\dots,i-1,i+1,\dots,d\}$, note that $\lambda_I=(\det\Lambda)\lambda_i^{-1}$ and $\langle U, B_I\rangle=\langle u,b_i\rangle^2$.
\end{proof}

In the sequel, we will use the following constant:
$$O_k:=\Ha^k(S^k)=\frac{2\pi^{k/2}}{\Gamma\left(\frac k2\right)},$$
$\Ha^k$ denotes the $k$-dimensional Hausdorff measure and $S^k$ is the unit sphere in $\R^{k+1}$.

Note that for the constant $\beta_{d,k}$ from \eqref{PKF}, another representation follows by using the Legendre duplication formula for the Gamma function:
\begin{equation} \label{beta}
\beta_{d,k}=\frac 1{\binom{d-1}{k}}\frac{\Gamma(\frac d2)}{\Gamma(\frac k2 +1)\Gamma(\frac{d-k}2)}.
\end{equation}

\section{A translative integral formula and flag measures}

We say that a subset $X\subset\R^d$ has \emph{locally positive reach} if $X$ is closed and $\reach(X,x)>0$ at any $x\in X$. Since the reach function is continuous in $X$, sets with locally positive reach behave locally as sets with positive reach (introduced by Federer, see \cite{Fed59}) and practically all local quantities and results known for sets with positive reach can be extended to this setting.

Let a subset $X\subset\R^d$ with locally positive reach be given. For any $x\in X$, the tangent (contingent) cone $\Tan (X,x)$ of $X$ at $x$ is a closed convex cone, and its dual cone $\Nor(X,x)$ is called \emph{normal cone} of $X$ at $x$. The \emph{unit normal bundle}
$$\nor X:=\{(x,u):\, x\in\partial X,\, u\in S^{d-1}\cap\Nor(X,x)\}$$
is a locally $(d-1)$-rectifiable subset of $\R^{2d}$ and at $\Ha^{d-1}$-almost all $(x,u)\in\nor X$, the $d-1$ principal curvatures $\kappa_1(x,u),\dots,\kappa_{d-1}(x,u)\in (-\infty,\infty]$ and corresponding principal directions $b_1(x,u),\dots,b_{d-1}(x,u)$ are defined so that $(b_1(x,u),\dots,b_{d-1}(x,u),u)$ is a positively oriented orthonormal basis od $\R^d$. (See \cite[Ch.~4]{RZ19} for details.) We will use the notation for subsets $I\subset\{0,\dots,d-1\}$
$$K_I(x,u):=\frac{\prod_{i\in I}\kappa_i(x,u)}{\prod_{i=1}^{d-1}\sqrt{1+\kappa_i^2(x,u)}},\quad
B_I(x,u):=\Lin\{b_i(x,u):\, i\in I\}.$$
(Here and in the sequel, we use the convention $\frac 1{\sqrt{1+\infty^2}}=0$, $\frac \infty{\sqrt{1+\infty^2}}=1$.) We shall use additionally the upper index $(X)$ if necessary.
The $k$th curvature-direction measure of $X$ ($k=0,1,\dots,d-1$) is a signed Radon measure in $\R^{2d}$ given by
$$\widetilde{C}_k(X,\cdot)=\frac 1{O_{d-1-k}}\int_{(\cdot)\cap\nor X}\sum_{|I|=d-1-k}K_I\, d\Ha^{d-1},$$
where $\sum_{|I|=j}$ denotes the summation over subsets of $\{1,\dots,d-1\}$ of cardinality $j$ (for $O_{d-1-k}$, see Preliminaries). The $k$th curvature measure $C_k(X,\cdot)$ and $k$th area measure $S_k(X,\cdot)$ are its marginals:
$$C_k(X,\cdot)=\widetilde{C}_k((\cdot)\times\R^d),\quad 
S_k(X,\cdot)=\widetilde{C}_k(\R^d\times(\cdot)).$$
For $k=d$ we set
$$C_d(X,\cdot)=\cL^d((\cdot)\cap X).$$

Let now two sets $X,Y\subset\R^d$ with locally positive reach be given. We say that $X$ and $Y$ \emph{touch} if there exists $(x,u)\in\nor X$ such that $(x,-u)\in\nor Y$. We will assume that

\begin{equation} \label{touch}
\cL^d\{z\in\R^d:\, X\text{ and }Y+z\text{ touch}\}=0.
\end{equation}

The \emph{mixed curvature measure} of $X,Y$ and orders $1\leq k,l\leq d-1$, $d\leq k+l$, is a signed measure on $\R^{2d}$ defined as
\begin{align*}
C_{k,l}(X,Y,A)=&\int_{\nor X\times\nor Y}\1_A(x,y)F_{k,l}(\angle(u,v))
\sum_{|I|=k^*}\sum_{|J|=l^*}K_I^{(X)}(x,u)K_J^{(Y)}(y,v)\\
&\times\left\|B_I^{(X)}(x,u)\wedge u\wedge B_J^{(Y)}(y,v)\wedge v\right\|^2\Ha^{2d-2}(d(x,u,y,v)),
\end{align*}
(recall that $k^*=d-1-k$), where 
$$F_{k,l}(\theta)=\frac 1{O_{2d-k-l-1}}\frac\theta{\sin\theta}
\int_0^1\left(\frac{\sin((1-t)\theta)}{\sin\theta}\right)^{d-1-k}
\left(\frac{\sin(t\theta)}{\sin\theta}\right)^{d-1-l}\, dt$$
if $\theta\in (0,\pi)$, $F_{k,l}(0):=\lim_{\theta\to 0+}F_{k,l}(\theta)$ and $F_{k,l}(\pi):=0$. (We will write $F_k(\cdot)$ in place of $F_{k,d-k}(\cdot)$ for brevity.)
In order that $C_{k,l}(X,Y,\cdot)$ is well-defined, we have to assume that its total variation measure
\begin{equation}  \label{bound} 
|C_{k,l}|(X,Y,\cdot)\text{ is locally finite}.
\end{equation}

The following version of the translative intersection formula is a special case of \cite[Theorem~6.10]{RZ19}.
\begin{theorem}  \label{T_transl}
Let $X,Y\subset \R^d$ be closed sets with locally positive reach satisfying \eqref{touch} and \eqref{bound}, and let $h$ be a nonnegative measurable function on $\R^d\times\R^d$ such that $\spt h\cap((X-Y)\times X)$ is bounded. Then
$$\int_{\R^d} h(z,x)V_0(X\cap(Y+z),dx)\, dz=\sum_{k=0}^d \int h(x-y,x)\,C_{k,d-k}(X,Y, d(x,y)).$$
\end{theorem}

Recall that given $0\leq k\leq d-1$, the $k$-\emph{flag space} $F^\perp(d,k)$ is the set of all pairs $(u,U)\in S^{d-1}\times G(d,k)$ such that $u\perp U$. The section of $F^\perp(d,k)$ at $u\in S^{d-1}$ will be denoted by $G^{u^\perp}(d-1,k)$ (it is, in fact, the Grassmannian in $u^\perp\simeq \R^{d-1}$). The integration over a Grassmannian $G(d,k)$ with respect to the invariant probability measure will be denoted simply by $dU$ ($U\in G(d,k)$).

Let $X\subset\R^d$ have locally positive reach and let $k\in\{0,\dots,d-1\}$ be given. The $k$th \emph{flag measure} of $X$ is the Radon signed measure on $\R^d\times F^\perp(d,k^*)$ given by
\begin{align}  \label{def_Gamma}
\int &g(x,u,U)\,\Gamma_k(X,d(x,u,U))\\
&=\gamma_{d,k}\int_{\nor X}\sum_{|I|=k^*}K_I(x,u)\int_{G^{u^\perp}(d-1,k^*)}g(x,u,U)\langle U,B_I(x,u)\rangle^2\, dU\, \Ha^{d-1}(d(x,u)),  \nonumber
\end{align}
where $\gamma_{d,k}=\binom{d-1}{k}/O_{k^*}$. Its marginal distribution on $F^\perp(d,k^*)$ is denoted by
$$\Omega_k(X,\cdot)=\Gamma_k(X,\R^d\times(\cdot)).$$

By \cite[Proposition~2]{HRW13}, there exists a smooth function $\varphi_k$ on $F^\perp(d,k^*)\times F^\perp(d,(d-k)^*)$ such that
\begin{equation}  \label{E_varphi}
\int_{F^\perp(d,k^*)}\int_{F^\perp(d,(d-k)^*)}\langle A,U\rangle^2\varphi_k(u,U,v,V)\langle B,V\rangle^2\, dV\, dU=\frac{1}{\gamma_{d,k}\gamma_{d,d-k}}\left\|A\wedge u\wedge B\wedge v\right\|^2
\end{equation}
for any $(u,A)\in F^\perp(d,k^*)$ and $(v,B)\in F^\perp(d,(d-k)^*)$. 
We will not state here the explicit form of $\varphi_k$, see \cite[Eq.~(17)]{HRW13}; an explicit form of the coefficients was obtained in \cite[Section~5]{GHHRW17}.

As a consequence we obtain
\begin{proposition} \label{P_mixed}
Two sets $X,Y\subset R^d$ with locally positive reach satisfy for $0\leq k\leq d-1$ and for any bounded Borel function $g$ with compact support
\begin{align*}
\int g(x,y)&\, C_{k,d-k}(X,d(x,y))=\\
&\int g(x,y) F_{k,d-k}(\angle(u,v))\varphi_k(u,U,v,V)\, (\Gamma_k(X,\cdot)\otimes\Gamma_{d-k}(Y,\cdot))(d(x,u,U,y,v,V)),
\end{align*}
provided that 
\begin{equation}  \label{integr}
|g(x,y)|F_k(\angle(u,v))|\varphi_k(u,U,v,V)|\text{ is locally } |\Gamma_k|(X,\cdot)\otimes|\Gamma_{d-k}|(Y,\cdot)\text{-integrable}.
\end{equation}
Also, \eqref{integr} implies that \eqref{bound} is satisfied with $k,d-k$.
\end{proposition}

The following lemma summarizes some sufficient conditions for the assumptions.

\begin{lemma}
Let $X,Y\subset\R^d$ have locally positive reach.
\begin{enumerate}
\item $X$ and $\rho Y$ satisfy \eqref{touch}, \eqref{bound} and \eqref{integr} for almost all rotations and/or reflections $\rho\in O(d)$.
\item $X$ and $Y$ satisfy \eqref{bound} and \eqref{integr} whenever for at least one of the sets $X$,$Y$, its $k$th area measure has locally bounded density w.r.t. $\Ha^{d-1}$. 
\end{enumerate}
\end{lemma}

\begin{proof}
For the validity of \eqref{touch} and \eqref{bound} under assumptions given in (i), see \cite[Proposition~6.13~(iv), Lemma~6.18]{RZ19}. 

In the proof of \cite[Theorem~2]{HRW13} it is shown that 
\begin{equation} \label{E_sph}
F_{k,d-k}(\angle(u,v))|\varphi_k(u,U,v,V)|\leq\operatorname{const}\cdot \sin^{3-d}(\angle(u,v)).
\end{equation}
Since $\int_{S^{d-1}}\sin^{3-d}(\angle(u,v))\, \Ha^{d-1}(dv)$ is a finite constant and $|\Gamma_k|(X,\cdot)$, $|\Gamma_{d-k}|(Y,\cdot)$ are locally finite, we obtain \eqref{integr} under (i) and (ii). Using the estimates from \cite[Lemma~6.15]{RZ19}, we obtain by similar reasoning that \eqref{bound} holds also under (ii).
\end{proof}

\section{Mean Euler characteristics for stationary random sets}
Let $Z\subset\R^d$ be a stationary random closed set (see \cite[Ch.~2]{SW08}). Since the family of sets with locally positive reach is a Borel subset of the family of closed sets (cf.\ \cite{Z84}), we can consider a a stationary random closed set $Z$ whose realizations have locally positive reach.

The $k$th curvature measure of $Z$, $C_k(Z,\cdot)$, is then a random measure in $\R^d$. Recall that we assume that  
$\E C_k(Z,\cdot)$ is locally finite (see \eqref{fin_dens}) and the curvature densities are defined by
$$\E C_k(Z,\cdot)=\overline{V}_k(Z)\cL^d(\cdot),\quad k=1,\dots,d.$$
If we consider the $k$th flag measure $\Gamma_k(Z,\cdot)$ of $Z$, notice that its marginal is $C_k(Z,\cdot)$. By disintegration of mean of the curvature-direction measure, we obtain that there exists a probability measure $\overline{S}_k(Z,\cdot)$ on $S^{d-1}$ such that
\begin{equation} \label{E_disint}
\E\int h(x,u)\,\widetilde{C}_k(Z,d(x,u))=\int_{\R^d}\int_{S^{d-1}}h(x,u)\, \overline{S}_k(Z,du)\, dx
\end{equation}
for any nonnegative measurable function $h$ on $\R^{2d}$.
Similarly, concerning the specific flag measure of order $k$,
we obtain that there exists a probability measure $\overline{\Omega}_k(Z,\cdot)$ on $F^\perp(d,k^*)$ such that
\begin{equation} 
\E\int g(x,u,U)\,\Gamma_k(Z,d(x,u,U))=\int_{\R^d}\int_{F^\perp(d,k^*)}g(x,u,U)\, \overline{\Omega}_k(Z,d(u,U))\, dx
\end{equation}
for any nonnegative measurable function $g$ on $\R^d\times F^\perp(d,k^*)$. Clearly, 
$$\overline{S}_k(Z,\cdot)=\overline{\Omega}_k(Z,(\cdot)\times G(d,k^*)).$$

We are now in position to prove our first main result.

\begin{proof}[Proof of Theorem~\ref{T1}]
Using the stationarity of $Z$ and Fubini theorem we obtain from \eqref{touch2} 
$$0=\int_{\R^d}\Pr[Z+x\text{ and }W\text{ touch}]\, dx
=\E\cL^d\{x\in\R^d:\, Z+x\text{ and }W\text{ touch}\},$$
hence \eqref{touch} holds for $Z,W$ almost surely. Also, Fubini yields that \eqref{integr_E} implies that $Z,W$ satisfy \eqref{integr} almost surely.

Let $B\subset\R^d$ be a fixed Borel set of unit volume. Theorem~\ref{T_transl} with $h(z,x)=\1_B(z)$ yields 
$$\int_B\chi((Z+z)\cap W)\, dz=\sum_{k=0}^d\int \1_B(x-y)\, C_{k,d-k}(Z,W,d(x,y))$$
almost surely. Applying 
Proposition~\ref{P_mixed}, we get almost surely
\begin{align*}
\int_B&\chi((Z+z)\cap W)\, dz\\
&=\int_W C_0(Z,B+y)\, dy+\int C_d(Z,B+y)\, C_0(W,dy)\\
&+\sum_{k=1}^{d-1}\iint \1_B(x-y)F_k(\angle(u,v))\varphi_k(u,U,v,V)\, \Gamma_k(Z,d(x,u,U)\, \Gamma_{d-k}(W,d(y,v,V).
\end{align*}
Now we use Fubini theorem (its use is justified by \eqref{integr_E}), stationarity of $Z$ and \eqref{E_disint2} and get
the desired formula for $\E\chi(Z\cap W)$.

Using formula \eqref{E_sph}, we easily find that \eqref{suffic} guarantees \eqref{integr_E}. Finally, if $\E|S_k|(Z,\cdot)$ has density $f$ on $S^{d-1}$ bounded by a constant $K>0$ then
$$\E\int \sin^{3-d}\angle(u,v)\, |S_k|(Z,du)\leq K\,
\E\int_{S^{d-1}} \sin^{3-d}\angle(u,v)\, \Ha^{d-1}du$$
and since the last integral is a finite constant independent of $v$, \eqref{integr_E} follows.
\end{proof}

\section{Excursion sets of stationary random fields}

Let $\xi:\R^d\to\R$ be a stationary random field and assume that $\xi$ is $C^{1,1}$ smooth almost surely. If $\alpha\in\R$ is a regular value of $\xi$ a.s. (i.e., if the gradient $\nabla\xi(x)\neq 0$ whenever $\xi(x)=\alpha$), the \emph{excursion set} of $\xi$ at $\alpha$,
$$Z_\alpha:=\{x\in\R^d:\, \xi(x)\geq\alpha\},$$
is a stationary random closed set with $C^{1,1}$ boundary (hence, with locally positive reach) almost surely. By the Rademacher theorem, the second order partial derivatives exist almost everywhere; we use the notation $\nabla^2\xi(x)$ for the Hessian matrix of second order derivatives of $\xi$ at $x$.

We assume in the sequel that 
\begin{equation} \label{Ass_dens}
\xi(0)\text{ is absolutely continuous with density }\varphi,
\end{equation}  
and that
\begin{equation} \label{Ass_sing}
\text{Almost all }\alpha\text{ are regular values of }\xi\text{ almost surely.}
\end{equation}

In order to apply the formula from Theorem~\ref{T1} for $Z_\alpha$, we need to obtain the specific flag measure $\overline{\Omega}_k(Z_\alpha,\cdot)$. 
In the smooth case, flag measures can be defined by integration over the boundary instead of the normal bundle; using the area formula for the mapping $(x,u)\mapsto x$ from $\nor Z_\alpha$ to $\partial Z_\alpha$ (see \cite[\S3.2.22]{Fed69}), we obtain from \eqref{def_Gamma} 
\begin{align}   \label{Egamma}
\int &\1_A(x)g(u,U)\,\Gamma_k(Z_\alpha,d(x,u,U))\\
&=\gamma_{d,k}\int_{A\cap \partial Z_\alpha}\sum_{|I|=k^*}k_I(x)\int_{G^{u(x)^\perp}(d-1,k^*)}g(u(x),U)\langle U,B_I(x)\rangle^2\, dU\, \Ha^{d-1}(dx),  \nonumber
\end{align}
where 
$$u(x):=\frac{\nabla\xi(x)}{|\nabla\xi(x)|}$$ 
is the (unique) unit outer normal vector to $Z_\alpha$ at $x\in\partial Z_\alpha$, $k_I(x)=\prod_{i\in I}\kappa_i(x)$ is the product of principal curvatures and $B_I(x)$ the linear span of the associated principal directions (as in \eqref{def_Gamma}, but now functions of solely $x$). 

\begin{lemma} \label{L_nabla}
Let $\xi$ be twice differentiable at $x\in\partial Z_\alpha$. Then for any $U\in G^{u(x)^\perp}(d-1,k^*)$,
$$\sum_{|I|=k^*}k_I(x)\langle U,B_I(x)\rangle^2=\frac{\left(-\nabla^2\xi(x)\right)[U]}{|\nabla\xi(x)|^{k^*}}.$$
\end{lemma}

\begin{proof}
Let $L:=-du(x):u(x)^\perp\to u(x)^\perp$ be the Weingarten mapping at $x\in\partial Z_\alpha$. Then, $Lb_i=\kappa_ib_i$ for the principal curvatures $\kappa_i$ and principal directions $b_i$, $i=1,\dots,d-1$. Thus, using Lemma~\ref{L_bilin},
$$\sum_{|I|=k^*}k_I(x)\langle U,B_I(x)\rangle^2=L[U].$$
Note that $du(x)$ agrees with the orthogonal projection onto $u(x)^\perp$ of 
$\frac{d\nabla\xi(x)}{|\nabla\xi(x)|}=\frac{\nabla^2\xi(x)}{|\nabla\xi(x)|}$. Hence, since $U\subset u^\perp$, we have $L[U]=(-\nabla^2(x))[U]\;|\nabla(x)|^{-k^*}$, and the assertion follows.
\end{proof}

\begin{proposition}   \label{main2}
Let $\xi$ be a stationary random field fulfilling \eqref{Ass_dens} and \eqref{Ass_sing}, and let $\alpha$ be its regular value a.s. Then for any nonnegative measurable function $g$ on $F^\perp(d,k^*)$ we have
\begin{align*}
\int g(u,U)\, &\overline{\Omega}_k(Z_\alpha,d(u,U))\\
&=\gamma_{d,k}\varphi(\alpha)\E_\alpha\left( |\nabla\xi(0)|^{2+k-d}\int_{G^{u(0)^\perp}(d-1,k^*)}g(u(0),U)(-\nabla^2\xi(0))[U]\, dU\right),
\end{align*}
provided that the expectation on the right hand side converges; here $\E_\alpha$ denotes the conditional expectation under condition $\xi(0)=\alpha$.
\end{proposition}

\begin{proof}
Lemma~\ref{L_nabla} applied to \eqref{Egamma} yields
\begin{align*}  
\int \1_A(x)&g(u,U)\,\Gamma_k(Z_\alpha,d(x,u,U))\\
&=\gamma_{d,k}\int_{A\cap \partial Z_\alpha}|\nabla\xi(x)|^{-k^*}\int_{G^{u(x)^\perp}(d-1,k^*)}g(u(x),U)(-\nabla^2\xi(x))[U]\, dU\, \Ha^{d-1}(dx).
\end{align*}
Let $A\subset\R^d$ be a Borel set of unit volume and $h$ a bounded measurable real function with bounded support.
Since $x\mapsto \xi(x)$ is locally Lipschitz with Jacobian $|\nabla(x)|$, the coarea formula (see \cite[\S3.2.22]{Fed69}) gives the following identity almost surely:
\begin{align*} 
\int_{\R}&\int_{A\cap\partial Z_\alpha}|\nabla\xi(x)|^{-k^*}\int_{G^{u(x)^\perp}(d-1,k^*)}g(u(x),U)(-\nabla^2\xi(x))[U]\,dU\, \Ha^{d-1}(dx)\, \varphi(\alpha)\,h(\alpha)d\alpha \\ 
&=\int_{A}|\nabla\xi(x)|^{2+k-d}\int_{G^{u(x)^\perp}(d-1,k^*)}g(u(x),U)(-\nabla^2\xi(x))[U]\,dU\, \varphi(\xi(x))\,h(\xi(x))\,\cL^d(dx).\\
\end{align*}
Thus, taking expectations and using \eqref{E_disint2} for the left hand side and the stationarity for the right hand side, we get
\begin{align*} 
\int\int &g(u,U)\overline{\Omega}_k(Z_\alpha,d(u,U)\, h(\alpha)\,\varphi(\alpha)d\alpha \\ 
&=\gamma_{d,k}\E |\nabla\xi(0)|^{2+k-d}\int_{G^{u(0)^\perp}(d-1,k^*)}g(u(0),U)(-\nabla^2\xi(0))[U]\,dU\, \varphi(\xi(0))h(\xi(0)).
\end{align*}
The assertion follows now from the definition of conditional expectation.
\end{proof}

\section{Gaussian random fields}  \label{GRF}
We will assume now that $\xi$ is a stationary \emph{Gaussian} random field with zero mean ($\E\xi(0)=0$). Thus, the distribution of $\xi$ is characterized by its covariance function $C(x)=\E\xi(0)\xi(x)$, $x\in\R^d$. We will assume that $\xi$ is $C^{1,1}$-smooth almost surely. This implies that the second order partial derivatives of $\xi$ exist in the $L^2$ sense (otherwise, the first order derivatives would be nondifferentiable almost everywhere and almost surely, see \cite[Theorem~4]{Cambanis}). This means that the covariance function $C$ has to be of class $C^4$, but this is not sufficient for the $C^2$ or $C^{1,1}$ smoothness of paths. In fact, it seems that $C^{1,1}$-smoothness implies already $C^2$-smoothness of the paths a.s., since boundedness a.s. of paths of second order partial derivatives (which are Gaussian random fields) imply already their smoothness a.s. (\cite[\S1.3]{AT07}). See \cite{A81, AT07} for a detailed treatment of continuity and differentiability of Gaussian processes.

A very important fact about Gaussian random fields is that the partial derivatives of first and second orders
$\xi_i(x):=\frac{\partial\xi}{\partial x_i}(x)$ and $\xi_{jk}(x):=\frac{\partial^2\xi}{\partial x_ix_j}(x)$, $x\in\R^d$, $1\leq i,j,k\leq d$, are again Gaussian; moreover, they are even jointly Gaussian together with $\xi(x)$. The partial derivatives have again zero mean and the covariances at $0$ are (see \cite[Sect.~5.5]{AT07})
\begin{align*}
&\E\xi(0)\xi_i(0)=0,\quad \E\xi_i(0)\xi_j(0)=-C_{ij}(0),\quad \E\xi(0)\xi_{jk}(0)=C_{jk}(0),\\
&\quad \E\xi_i(0)\xi_{jk}(0)=0,\quad \E\xi_{ij}\xi_{kl}=C_{ijkl}(0),\qquad 1\leq i,j,k,l\leq d
\end{align*}
(again, the lower indices at $C$ denote its partial derivatives). Due to the Gaussianity, this implies that the vectors
$$(\xi(0),\xi_{jk}(0),\,1\leq j,k\leq d)\text{ and }(\xi_i(0),\,1\leq i\leq d)$$
are independent. This means that also the conditional distributions of first and second partial derivatives are mutually independent under condition $\xi(0)=\alpha$, and the conditioning does not influence the first order partial derivatives. 
Consequently, we can evaluate in the formula in Proposition~\ref{main2} separately the mean value of $(-\nabla^2\xi(0))[U]$.

\begin{lemma}  \label{L_nabla2}
Assume that the gradient $\nabla\xi(0)$ has nondegenerate distribution. Then, for any $U\in G(d,k^*)$, 
$$\E_\alpha(-\nabla^2\xi(0))[U]=\sigma^{-k^*}\Lambda[U]H_{k^*}\left(\frac\alpha\sigma\right),$$
where 
$$\Lambda:=\left(\E\frac{\partial\xi}{\partial x_i}(0) \frac{\partial\xi}{\partial x_j}(0)\right)_{i,j=1}^d$$
and 
$$H_n(t):=n!\sum_{j=0}^{\lfloor n/2\rfloor}\frac{(-1)^jt^{n-2j}} {j!(n-2j)!2^j},\quad t\in\R,$$
is the $n$th Hermite polynomial. Further, we have
$$\E_\alpha\left|(-\nabla^2\xi(0))[U]\right|\leq\const(d,k,C)$$
(the constant on the right hand side depends on the covariance function $C$, but not on $U$). 
\end{lemma}

\begin{proof}
Let $\{u_1,\dots,u_{k^*}\}$ be an orthonormal basis of a subspace $U\in G(d,k^*)$.
Since $\Lambda$ is symmetric and positive definite by assumptions, the matrix 
$\Lambda_U=(\lambda_{ij})_{i,j=1}^{k^*}$ with $\lambda_{ij}:=\langle u_i,\Lambda u_j\rangle$ is symmetric and positive definite as well. Hence, there exists a symmetric matrix $\Lambda_U^{-1/2}$ such that $\Lambda_U^{-1/2}\Lambda_U\Lambda_U^{-1/2}=I$. By definition, $(-\nabla^2\xi(x))[U]=(-1)^{k^*}\det Y$ for a random matrix $Y$ with zero-mean Gaussian entries
$$Y_{ij}=\frac{\partial^2\xi}{\partial u_i\partial u_j}(0),\quad 1\leq i,j\leq k^*$$
with $\E \;\xi(0)Y_{ij}=-\lambda_{ij}$ and $\E\; Y_{ij}Y_{lm}= C^U_{ijlm}(0)$ (the fourth order partial derivative of $C$ at $0$ in directions $u_i,u_j,u_l,u_m$). Consequently, the conditional moments are
$$\E_\alpha Y_{ij}=\sigma^{-2}\lambda_{ij}\alpha,\quad
\cov_\alpha(Y_{ij},Y_{lm})=C^U_{ijlm}(0)-\sigma^{-2}\lambda_{ij}\lambda_{lm}.$$
Proceeding as in \cite[Proof of Lemma~11.7.1]{AT07}, we consider the transformation
$$Z:=\sigma \Lambda_U^{-1/2}Y\Lambda_U^{-1/2}$$
and find that
$$\E_\alpha Z_{ij}=\sigma^{-1}\alpha \delta_{ij},\quad \cov_\alpha(Z_{ij},Z_{lm})=\cE(i,j,l,m)-\delta_{ij}\delta_{lm}$$
with a function $\cE(i,j,l,m)$ symmetric in $i,j,l,m$ (a linear combination of fourth order partial derivatives of $C$). Now \cite[Corollary~11.6.3]{AT07} yields
$$\E_\alpha\det Z=(-1)^{k^*}H_{k^*}\left(\frac\alpha\sigma\right),$$
and using the definition of $Z$, we obtain the first result.

Note that the coefficients $\lambda_{ij}$ are bounded in absolute value by the norm of the matrix 
$\left(\frac{\partial^2 C}{\partial x_i\partial x_j}(0)\right)$. Hence, all the means and variances of $Y_{ij}$ are bounded and, consequently, the $k^*$th absolute moment of all $Y_{ij}$ is bounded by that of a single Gaussian variable $Y$ whose mean and variance do not depend on $U$. By the generalized H\"older inequality,
$$\E_\alpha\left|(-\nabla^2\xi(x))[U]\right|\leq k^*!\,\E_\alpha |Y|^{k^*}<\infty,$$
which completes the proof.
\end{proof}

As concerns Theorem~\ref{T2}, we start by proving the first (regularity) statement and the ``touching condition'' \eqref{touch2}. We will use the following auxiliary lemma.

\begin{lemma}  \label{L_aux}
Let $B\subset\R^d$ be such that for some $\alpha,p>0$, $\Ha^p(B)<\infty$ and 
$$\Ha^p(B\cap B(x,\delta))\geq \alpha\delta^p, \quad x\in B,\, 0<\delta<1.$$ 
Let $(\eta(x):\, x\in B)$ be a random field with values in $\R^k$ which is a.s.\ Lipschitz, and assume that there exist $\beta>0$ and $q>p$ such that
$$\Pr[|\eta(x)|\leq\delta]\leq\beta\delta^q,\quad x\in B,\, 0<\delta<1.$$
Then
$$\Pr\left[\exists x\in B:\, \eta(x)=0\right]=0.$$
\end{lemma}

\begin{remark}
A set $B$ with the property above is called \emph{Ahlfors lower $p$-regular}.
\end{remark}

\begin{proof}
By assumption, there exists a random variable $C>0$ such that
$$|\eta(x)-\eta(y)|\leq C |y-x|\quad \text{a.s.}$$
Thus, if $\eta(x)=0$ for some $x\in B$ then $|\eta(y)|\leq C\delta$ for all $y\in B\cap B(x,\delta)$ and, hence, for any $c>0$ and $0<\delta<c^{-1}$,
$$
\Pr([C\leq c]\cap[\exists x\in B:\, \eta(x)=0])
\leq \Pr\left[\Ha^p\{ y\in B:\, |\eta(y)|\leq c\delta\}\geq\alpha\delta^p\right].$$
Applying the Chebyshev's inequality and Fubini theorem we get
\begin{align*}
\Pr([C\leq c]\cap[\exists x\in B:\, \eta(x)=0])&\leq \alpha^{-1}\delta^{-p}\,\E \Ha^p\{ y\in B:\, |\eta(y)|\leq c\delta\}\\
&=\alpha^{-1}\delta^{-p}\int_B\Pr[|\eta(y)|\leq c\delta]\,\Ha^p(dy)\\
&\leq\alpha^{-1}\delta^{-p}\,\Ha^p(B)\beta(c\delta)^q.
\end{align*}
The last expression tends to zero as $\delta\to 0$, hence,
$$\Pr([C\leq c]\cap[\exists x\in B:\, \eta(x)=0])=0.$$
Letting $c\to\infty$, we obtain the result.
\end{proof}

\begin{lemma}
Let $\xi$ be as in Theorem~\ref{T2}. Then for any $\alpha\in\R$ and any compact set $W\subset\R^d$ with positive reach,
\begin{enumerate}
\item[(i)] $\alpha$ is a regular value of $\xi$ almost surely,
\item[(ii)] $\Pr[Z_\alpha\text{ and }W\text{ touch}]=0$.
\end{enumerate}
\end{lemma}

\begin{proof}
(i) Set $\eta(x)=(\xi(x)-\alpha,\nabla\xi(x))$, $x\in\R^d$. $\eta$ is a stationary random field and since $\xi(0)$ and $\nabla\xi(0)$ are independent, their joint distribution has a bounded density with respect to $\cL^{d+1}$. Thus, $\eta$ fulfills the assumption of Lemma~\ref{L_aux} with $q=d+1$, and applying this lemma with $B=[0,1]^d$ and $p=d$, we obtain that
$$\Pr\left[\exists x\in [0,1]^d:\, \xi(x)=\alpha,\, \nabla(x)=0\right] =0.$$
Due to stationarity, this already implies (i).

(ii) Since $\nor W$ is a compact $(d-1)$-dimensional Lipschitz manifold (see \cite[Corollary~4.22]{RZ19}), it is Ahlfors lower $(d-1)$-regular (see \cite[Proposition~1.12]{RZ19}). Set $\eta(x,v):=(\xi(x)-\alpha,u(x)+v)$, $x\in \nor W$. The random unit vector $u(x)$ has a bounded density w.r.t. $\Ha^{d-1}$ on the unit sphere (cf.\ the proof of Theorem~\ref{T2} below where this density is obtained explicitly) and since $u(x)$ and $\xi(x)$ are independent, $\eta$ fulfills the assumption of Lemma~\ref{L_aux} with $q=d$. Thus, this lemma gives that 
$$\Pr[\exists (x,v)\in\nor W:\, \xi(x)=\alpha, u(x)+v=0]=0,$$
which is exactly the assertion (ii).
\end{proof}

Now we can prove Theorem~\ref{T2} which yields an expression of the density of the flag measure of the excursion set by its density with respect to the invariant measure $\mu_k$ on $F^\perp(d,k^*)$.

\begin{proof}[Proof of Theorem~\ref{T2}]
Since $(\xi(0),\nabla^2\xi(0))$ and $\nabla\xi(0)$ are independent, the formula from Proposition~\ref{main2} can be rewritten using Lemma~\ref{L_nabla2} as
\begin{align*}
\int &g(u,U)\, \overline{\Omega}_k(Z_\alpha,d(u,U))\\
&=\gamma_{d,k}\frac 1{\sqrt{2\pi}\sigma}e^{-\frac{\alpha^2}{2\sigma^2}}\E\left( |\nabla\xi(0)|^{2+k-d}\int_{G^{u(0)^\perp}(d-1,k^*)}g(u(0),U)\E_\alpha(-\nabla^2\xi(0))[U]\, dU\right)\\
&=\gamma_{d,k}\frac 1{\sqrt{2\pi}\sigma}e^{-\frac{\alpha^2}{2\sigma^2}}
\sigma^{-k^*}H_{k^*}\left(\frac\alpha\sigma\right)\E\left( |\nabla\xi(0)|^{2+k-d}\int_{G^{u(0)^\perp}(d-1,k^*)}g(u(0),U)\Lambda[U]\, dU\right).
\end{align*}
(The convergence of the integral was guarantied by the inequality in Lemma~\ref{L_nabla2}.)
In order to express the mean value, we compute the joint density of $(|\nabla\xi(0)|,u(0))$. The probablity density function of $\nabla\xi(0)$ is
$$f(x):=\frac 1{(2\pi)^{d/2}\sqrt{\det\Lambda}}\exp\left(-\frac 12 x^T\Lambda^{-1}x\right).$$
The Lipschitz bijection $(r,u)\mapsto ru$ from $(0,\infty)\times S^{d-1}$ onto $\R^d\setminus\{0\}$ has Jacobian $r^{d-1}$, hence, by the area formula, the joint density of $(|\nabla\xi(0)|,u(0))$ is
$$F(r,u)=\frac 1{(2\pi)^{d/2}\sqrt{\det\Lambda}}r^{d-1}\exp\left(-\frac{r^2}2 u^T\Lambda^{-1}u\right),\quad (r,u)\in (0,\infty)\times S^{d-1}.$$
A direct computation yields
$$\int_0^\infty r^{2+k-d}F(r,u)\, dr=\frac{2^{k/2}\Gamma\left(\tfrac{k}2+1\right)}{(2\pi)^{d/2}\sqrt{\det\Lambda}}\left(u^T\Lambda^{-1}u\right)^{-\frac k2-1}$$
and, since (using \eqref{beta})
$$\gamma_{d,k}\Gamma(\tfrac k2+1)=\frac 1{\beta_{d,k}}\frac{\Gamma(\frac d2)}{2\pi^{\frac{d-k}2}}=\frac 1{\beta_{d,k}}\frac{\pi^{\frac k2}}{O_{d-1}},$$
we arrive at,
\begin{align*}
\int g(u,U)\, \overline{\Omega}_k(Z_\alpha,d(u,U))
&= \frac{\beta_{d,k}^{-1}\;(2\pi)^{\frac{k-d-1}2}}{O_{d-1}\sigma^{d-k}\sqrt{\det\Lambda}}e^{-\frac{\alpha^2}{2\sigma^2}}H_{k^*}\left(\frac\alpha\sigma\right)\\
&\times\int_{S^{d-1}}\left(u^T\Lambda^{-1}u\right)^{-\frac k2-1}\int_{G^{u^\perp}(d-1,k^*)}g(u,U)\Lambda[U]\, dU\, \Ha^{d-1}(du),
\end{align*}
which implies the assertion.
\end{proof}

\section{Polytopal window}
The aim of this section is to evaluate the formula of Corollary~\ref{C2} in the case of a convex polytope $F$ and, in particular, to prove Corollary~\ref{C4}.

Let $W\subset\R^d$ be a convex polytope. Its unit normal bundle can be represented as
$$\nor W=\bigcup_{l=0}^d \sum_{F\in\cF_l(W)} (\relint F)\times\Gamma_F,$$
where $\cF_l(W)$ is the family of all $l$-dimensional faces of $W$ and given a face $F$, $\Gamma_F$ denotes the set of all unit outer normals to $W$ at (any) point of $\relint F$ (cf.\ \cite[Section~4.2]{Sch14}). Hence, using \eqref{def_Gamma}, the $l$th flag measure of $W$ is given by
$$\int g(v,V)\,\Omega_l(W,d(v,V))=\gamma_{d,l}\sum_{F\in\cF_l(W)}|F|\int_{\Gamma_F}\int_{G^{v^\perp}(d-1,l^*)}\langle F^\perp\cap v^\perp,V\rangle^2\, dV\,dv, $$
where we write for brevity $|F|:=\Ha^l(F)$, $dv:=\Ha^{d-1}(dv)$ and $F^\perp\in G(d,d-l)$ is the $(d-l)$-subspace perpendicular to $F$.

Let further $\xi$ be a centred stationary Gaussian random field as in Theorem~\ref{T2} and $Z_\alpha$ its excursion set. Corollary~\ref{C2} implies
$$\E\chi(Z_\alpha\cap W)=\sum_{k=0}^{d-1}P_k+\overline{V}_d(Z_\alpha),$$
where
\begin{align*}
P_k:=&\gamma_{d,d-k}\sum_{F\in\cF_{d-k}(W)}|F| \int_{\Gamma_F}\int_{S^{d-1}}F_k(\angle(u,v))\int_{G^{v^\perp}(d-1,k-1)}\int_{G^{u^\perp}(d-1,d-1-k)}\\
&\langle F^\perp\cap v^\perp,V\rangle^2\varphi_k(u,U,v,V)\,\mathbf{c}\frac{\Lambda[U]}{\sqrt{\det\Lambda}}\, dU\, dV\, (u^T\Lambda^{-1}u)^{-\frac k2-1}\, du\, dv,
\end{align*}
with a constant
$$\mathbf{c}:=\frac{\beta_{d,k}^{-1}(2\pi)^{\frac{k-d-1}2}}{O_{d-1}\sigma^{d-k}}e^{-\frac{\alpha^2}{2\sigma^2}}H_{k^*}\left(\frac\alpha\sigma\right).$$
Let $\lambda_i,b_i$ be the eigenvalues and eigenvectors of $\Lambda$, $i=1,\dots, d$. Lemma~\ref{L_bilin} implies that
$$\Lambda[U]=\sum_{|I|=k^*}\lambda_I\langle U,B_I\rangle^2=
\sum_{|I|=k^*}\lambda_I|p_{B_I^\perp}u|^2\langle U,B_I^*\rangle^2,$$
where $B_I^*:=p_{u^\perp}(B_I)$ (orthogonal projection of $B_I$ into $u^\perp$). Thus, Equation~\eqref{E_varphi} implies
\begin{align*}
P_k=\frac{\mathbf{c}}{\gamma_{d,k}}\frac 1{\sqrt{\det\Lambda}} \sum_{F\in\cF_{d-k}(W)}&|F| \int_{\Gamma_F}\int_{S^{d-1}} \\
&\sum_{|I|=k^*}\lambda_I \|F^\perp\wedge B_I^*\wedge u\|^2 |p_{B_I^\perp}u|^2 \frac{F_k(\angle(u,v))}{(u^T\Lambda^{-1}u)^{\frac k2+1}} \, du\, dv.
\end{align*}
Note that $\|F^\perp\wedge B_I^*\wedge u\|^2 |p_{B_I^\perp}u|^2=\|F^\perp\wedge B_I\wedge u\|^2$, and, using Lemma~\ref{L_bilin} and \eqref{E_u^perp}, we get
\begin{align*}
\sum_{|I|=k^*}\lambda_I \|F^\perp\wedge B_I\wedge u\|^2
=&\sum_{|I|=k^*}\lambda_I\langle F_0\cap u^\perp,B_I\rangle^2|p_{F_0}u|^2\\
=&\Lambda[F_0\cap u^\perp]|p_{F_0}u|^2\\
=&\Lambda[F_0]\, (p_{F_0}u)^T\Lambda^{-1}(p_{F_0}u),
\end{align*}
where $F_0$ denotes the linear hull of the face $F$ shifted to the origin and $p_{F_0}$ is the orthogonal projection to $F_0$. Consequently,
$$P_k=\frac{\mathbf{c}}{\gamma_{d,k}} \sum_{F\in\cF_{d-k}(W)}|F| \frac{\Lambda[F_0]}{\sqrt{\det\Lambda}} \int_{\Gamma_F}\int_{S^{d-1}} F_k(\angle(u,v)) \frac{(p_{F_0}u)^T\Lambda^{-1}(p_{F_0}u)}{(u^T\Lambda^{-1}u)^{\frac k2+1}}\, du\, dv.$$
The last expression can be simplified significantly in the special case when $W$ is a \emph{zonotope}. We will assume that $W$ has a vertex at the origin, which means that there exists a finite family of vectors $v_1,\dots,v_m\in\R^d$ such that $W$ can be expressed as the Minkowski sum of segments
$$W=[0,v_1]\oplus\dots\oplus[0,v_m].$$
Let $\cF_l^0(W)$ denote the family of all facets $F\in\cF_l(W)$ containing the origin. The crucial observation is that, given $F\in\cF_l^0(W)$, the sets $\Gamma_{F'}$ for facets $F'\in\cF_l(W)$ parallel with $F$ (denoted $F'\| F$) form a partition of the unit sphere $S^{d-l-1}$ in $F^\perp$. Hence, by \cite[Eq.~(6.15)]{RZ19}, 
$$\sum_{{F'\in\cF_{d-k}(W)}\atop{F'\|F}}\int_{\Gamma_{F'}}F_k(\angle(u,v))\, dv=\frac{|p_{F_0}u|^{k-d}}{O_{d-k-1}}.$$
Consequently, we obtain for a zonotope $W$
\begin{equation} \label{zonotope}
P_k=\frac{\mathbf{c}}{\binom{d-1}{k}} \sum_{F\in\cF_{d-k}^0(W)} |F| \frac{\Lambda[F]}{\sqrt{\det\Lambda}} \int_{S^{d-1}} \frac{(p_{F_0}u)^T\Lambda^{-1}(p_{F_0}u)}{|p_{F_0}u|^{d-k}(u^T\Lambda^{-1}u)^{\frac k2+1}}\, du.
\end{equation}
In order to evaluate the last integral, we use the method of Gaussian random vectors as in the proof of Theorem~2. Let $X$ be a $d$-dimensional Gaussian random vector with mean $0$ and variance matrix $\Lambda$. Then,
\begin{align*}
\E\frac{(p_{F_0}X)^T\Lambda^{-1}(p_{F_0}X)}{|p_{F_0}X|^{d-k}}
=&\int_{R^d}\frac{(p_{F_0}x)^T\Lambda^{-1}(p_{F_0}x)}{|p_{F_0}x|^{d-k}}\frac 1{(2\pi)^{\frac d2}\sqrt{\det\Lambda}}e^{-\frac{x^T\Lambda^{-1}x}2}\, dx\\
=&\frac 1{(2\pi)^{\frac d2}\sqrt{\det\Lambda}} \int_{S^{d-1}}\int_0^\infty 
\frac{r^2(p_{F_0}u)^T\Lambda^{-1}(p_{F_0}u)}{r^{d-k}|p_{F_0}u|^{d-k}}e^{-\frac{r^2}2u^T\Lambda^{-1}u}\, r^{d-1}\, dr\, du\\
=&\frac 1{(2\pi)^{\frac d2}\sqrt{\det\Lambda}} \int_{S^{d-1}}
\frac{(p_{F_0}u)^T\Lambda^{-1}(p_{F_0}u)}{|p_{F_0}u|^{d-k}}
\frac{2^{\frac k2}\Gamma(\tfrac k2+1)}{(u^T\Lambda^{-1}u)^{\frac k2+1}}\, du\\
=&\frac{\Gamma(\tfrac k2+1)}{2^{\frac{d-k}2}\pi^{\frac d2}\sqrt{\det\Lambda}} \int_{S^{d-1}} \frac{(p_{F_0}u)^T\Lambda^{-1}(p_{F_0}u)}{|p_{F_0}u|^{d-k}(u^T\Lambda^{-1}u)^{\frac k2+1}}\, du.
\end{align*}
Denoting $Y:=p_{F_0}X$, which is again Gaussian with zero mean and variance matrix $\Lambda_{F_0}$, and $Z:=\Lambda_{F_0}^{-\frac 12}Z$, which is Gaussian with unit variance matrix, we compute similarly
\begin{align*}
\E\frac{(p_{F_0}X)^T\Lambda^{-1}(p_{F_0}X)}{|p_{F_0}X|^{d-k}}
=&\E\frac{Y^T\Lambda_{F_0}^{-1}Y}{|Y|^{d-k}}
=\E\frac{|Z|^2}{(Z^T\Lambda_{F_0} Z)^{\frac{d-k}2}}\\
=&\frac 1{(2\pi)^{\frac{d-k}2}}\int_{\R^{d-k}}\frac{|z|^2}{(z^T\Lambda_{F_0} z)^{\frac {d-k}2}}e^{-\frac{|z|^2}2}\, dz\\
=&\frac 1{(2\pi)^{\frac{d-k}2}}\int_{S^{d-k-1}}\int_0^\infty\frac{r^2e^{-r^2/2}}{r^{d-k}}r^{d-k-1}\, dr\frac{du}{(u^T\Lambda_{F_0} u)^{\frac{d-k}2}}\\
=&\frac 1{(2\pi)^{\frac{d-k}2}}O_{d-k-1}\sqrt{\det\Lambda_{F_0}^{-1}}.
\end{align*}
The last equality follows by expressing the total density of a Gaussian vector with zero mean and variance matrix $\Lambda_{F_0}^{-1}$ in spherical coordinates:
\begin{align*}
1=&\int_{\R^{d-k}}\frac 1{(2\pi)^{\frac{d-k}2}\sqrt{\det\Lambda_{F_0}^{-1}}}e^{-\frac{z^T\Lambda_{F_0}z}2}\, dz\\
=&\frac 1{(2\pi)^{\frac{d-k}2}\sqrt{\det\Lambda_{F_0}^{-1}}}\int_{S^{d-k-1}}\int_0^\infty e^{-\frac{r^2}2{u^T\Lambda_{F_0}u}}\, r^{d-k-1}\, dr\, du\\
=&O_{d-k-1}^{-1}\sqrt{\det\Lambda_{F_0}}\int_{S^{d-k-1}}\frac{du}{(u^T\Lambda_{F_0} u)^{\frac{d-k}2}}.
\end{align*}
Putting the last two equalities together, we obtain
$$\frac{1}{\sqrt{\det\Lambda}} \int_{S^{d-1}} \frac{(p_{F_0}u)^T\Lambda^{-1}(p_{F_0}u)}{|p_{F_0}u|^{d-k}(u^T\Lambda^{-1}u)^{\frac k2+1}}\, du
=\frac{\pi^{\frac k2}O_{d-1-k}}{\Gamma(\frac k2+1)}\sqrt{\Lambda[F_0]}
=\frac{2\pi^{\frac d2}}{\Gamma(\frac k2+1)\Gamma(\frac{d-k}2)}\sqrt{\Lambda[F_0]}.$$
Inserting the last equality into \eqref{zonotope}, and using \eqref{beta}, we conclude with 
$$P_k= \frac{1}{(2\pi)^{\frac{d-k+1}2}\sigma^{d-k}}e^{-\frac{\alpha^2}{2\sigma^2}}H_{k^*}\left(\frac\alpha\sigma\right)\sum_{F\in\cF_{d-k}^0(W)} |F| \sqrt{\Lambda[F_0]}.$$
This proves Corollary~\ref{C4}.

\section{Curvature densities of excursion sets}

Let $\xi$ be a centred stationary Gaussian random field as in Theorem~\ref{T2}.  We evaluate the $k$th curvature density of the excursion set $Z_\alpha$. This will prove Corollary~\ref{C3} in the isotropic case.

By definition,
$$\overline{V}_k(Z_\alpha)=\int q_k\, d\mu_k.$$ 
Note that for any $U\in G^{u^\perp}(d-1,k^*)$, we have by Lemma~\ref{L_bilin}, $\Lambda[U]=\sum_{|I|=k^*}\lambda_I\langle U,B_I\rangle^2$, where 
$\lambda_I=\prod_{i\in I}\lambda_i$, $B_I=\Lin\{b_i:\, i\in I\}$ and $\lambda_i,b_i$, $i=1,\dots,d$, are the principal values and corresponding principal directions of $\Lambda$. Note further that
$$\langle U,B_I\rangle^2=|p_{B_I^\perp}u|^2\langle U,B_I^*\rangle^2,$$
where $B_I^*\in G^{u^\perp}(d-1,k^*)$ is the orthogonal projection of $B_I$ into $u^\perp$. Since
$$\int_{G^{u^\perp}(d-1,k^*)}\langle U,V\rangle^2\, dU=\binom{d-1}{k}^{-1},\quad V\in G^{u^\perp}(d-1,k^*)$$
(see \cite[p.~641]{HRW13}), we obtain
\begin{align*}
\int_{G^{u^\perp}(d-1,k^*)}\Lambda[U]\, dU &= \sum_{|I|=k^*}\lambda_I |p_{B_I^\perp}u|^2 \int_{G^{u^\perp}(d-1,k^*)}  \langle U,B_I^*\rangle^2\, dU\\
&=\binom{d-1}{k}^{-1}\sum_{|I|=k^*}\lambda_I\sum_{j\not\in I}\langle u,b_j\rangle^2\\
&=\sum_{j=1}^d\lambda_{(j)}\langle u,b_j\rangle^2,
\end{align*}
where 
$$\lambda_{(j)}:=\binom{d-1}{k}^{-1}\sum_{|I|=k^*,\, j\not\in I}\lambda_I.$$
Hence we obtain the formula
\begin{equation}  \label{Vk}
\overline{V}_k(Z_\alpha)=
\frac{\beta_{d,k}^{-1}\;(2\pi)^{\frac{k-d-1}2}}{O_{d-1}\sigma^{d-k}\sqrt{\det\Lambda}}e^{-\frac{\alpha^2}{2\sigma^2}}H_{k^*}\left(\frac\alpha\sigma\right)\int_{S^{d-1}}\frac{\sum_{j=1}^d\lambda_{(j)}\langle u,b_j\rangle^2}{(u^T\Lambda^{-1}u)^{\frac k2+1}}\,\Ha^{d-1}(du).
\end{equation}
In the isotropic case, we have 
$$\lambda_i=\lambda:=\E\left(\frac{\partial\xi}{\partial x_1}(0)\right)^2=-\frac{\partial^2C}{\partial x_1^2}(0), \quad i=1,\dots,d,$$ 
hence, $\det\Lambda=\lambda^d$, $\lambda_{(j)}=\lambda^{k^*}$, $j=1,\dots,d$, and $u^T\Lambda^{-1}u=\lambda^{-1}$, and \eqref{Vk} simplifies to the formula given in Corollary~\ref{C3}.


\begin{thebibliography}{99}

\bibitem{A81}
R.J. Adler: \emph{The geometry of random fields}. Wiley Series in Probability and Mathematical Statistics, J. Wiley \& Sons, Ltd., Chichester, 1981

\bibitem{AT07}
R.J. Adler, J.E. Taylor: \emph{Random Fields and Geometry.} Springer, 2007

\bibitem{Cambanis}
S. Cambanis: On some continuity and differentiability properties of paths of Gaussian processes.
\emph{J. Multivariate Anal.} {\bf 3} (1973), 420--434

\bibitem{Fed59} H. Federer: Curvature measures. \emph{Trans.\ Amer.\ Math.\ Soc.} {\bf 93} (1959), 418--491

\bibitem{Fed69} H. Federer: {\it Geometric Measure Theory.} Springer, Berlin 1969

\bibitem{GHHRW17}
P. Goodey, W. Hinderer, D. Hug, J. Rataj, W. Weil: A flag representation of projection functions. \emph{Adv.\ Geom.} {\bf 17} (2017), 303--322

\bibitem{HRW13} D. Hug, J. Rataj, W. Weil: A product integral representation of mixed volumes
of two convex bodies. \emph{Adv.\ Geom.} {\bf 13} (2013), 633--662

\bibitem{LR19}
R. Lachi\`eze-Rey: Bicovariograms and Euler characteristic of random fields excursions. \emph{Stochastic Process.\ Appl.} {\bf 129} (2019), 4687--4703

\bibitem{RZ95}
J. Rataj, M. Z\"ahle: Mixed curvature measures for sets of positive reach and a translative integral formula. \emph{Geom.\ Dedicata} {\bf 57} (1995), 259--283

\bibitem{RZ19}
J. Rataj, M. Z\"ahle: \emph{Curvature Measures of Singular Sets}. Springer, 2019

\bibitem{Sch14}
R. Schneider: \emph{Convex Bodies: The Brunn-Minkowski Theory}. Second expanded edition. Encyclopedia of Mathematics and its Applications, 151. Cambridge University Press, Cambridge, 2014

\bibitem{SW08}
R. Schneider, W. Weil: \emph{Stochastic and Integral Geometry.} Springer, New York, 2008

\bibitem{Z84}
M. Z\"ahle: Curvature measures and random sets I. \emph{Math.\ Nachr.} {\bf 119} (1984), 327--339


\end{thebibliography}
\end{document}